\documentclass[secthm,seceqn,amsthm,ussrhead,12pt]{amsart}
\usepackage[utf8]{inputenc}
\usepackage[english]{babel}
\usepackage{amssymb,amsmath,amsthm,amsfonts,xcolor,enumerate,hyperref,comment,longtable,cleveref}

\usepackage{times}
\usepackage{cite}
\usepackage{pdflscape}
\usepackage{ulem}
\usepackage[mathcal]{euscript}
\usepackage{tikz}
\usepackage{hyperref}
\usepackage{cancel}
\usepackage{stmaryrd}

\usetikzlibrary{arrows}

\newtheorem{theorem}{Theorem}
\newtheorem{lemma}[theorem]{Lemma}
\newtheorem{proposition}[theorem]{Proposition}

\newtheorem{definition}[theorem]{Definition}
\newtheorem{corollary}[theorem]{Corollary}

\usepackage{stmaryrd}
\usepackage{xcolor}

\begin{document}

\title{ $n$-Ary generalized Lie-type color algebras admitting a quasi-multiplicative basis}
\thanks{The first and fourth authors acknowledge financial assistance by the Centre for Mathematics of the University of Coimbra -- UID/MAT/00324/2013, funded by the Portuguese Government through FCT/MEC and co-funded by the European Regional Development Fund through the Partnership Agreement PT2020. Second and fourth authors are supported by the PCI of the UCA `Teor\'\i a de Lie y Teor\'\i a de Espacios de Banach', by the PAI with project numbers FQM298, FQM7156 and by the project of the Spanish Ministerio de Educaci\'on y Ciencia  MTM2016-76327C31P.  
The third author is supported by FAPESP 	17/15437-6. 
The fourth author acknowledges the Funda\c{c}\~{a}o para a Ci\^{e}ncia e a Tecnologia for the grant with reference SFRH/BPD/101675/2014.}

\author[E. Barreiro]{E. Barreiro}
\address{Elisabete Barreiro. CMUC, Universidade de Coimbra. Coimbra, Portugal.}
\email{mefb@mat.uc.pt}

\author[A.J. Calder\'on]{A.J. Calder\'on}
\address{A.J. Calder\'on. Departamento de Matem\'aticas, Universidad de C\'adiz. Puerto Real, C\'adiz, Espa\~na.}
\email{ajesus.calderon@uca.es}

\author[I. Kaygorodov]{I. Kaygorodov}
\address{Ivan Kaygorodov. CMCC, Universidade Federal do ABC. Santo Andr\'e, Brasil.}
\email{kaygorodov.ivan@gmail.com}

\author[J.M. S\'anchez]{J.M. S\'anchez}
\address{Jos\'e M. S\'anchez. Departamento de Matem\'aticas, Universidad de C\'adiz. Puerto Real, C\'adiz, Espa\~na.}
\email{txema.sanchez@uca.es}

\begin{abstract}
The class of generalized Lie-type color algebras contains the ones of generalized Lie-type  algebras, of $n$-Lie algebras and superalgebras, commutative Leibniz $n$-ary algebras and superalgebras, among others. 
We focus on the class of generalized Lie-type color algebras $\frak L$ admitting a quasi-multiplicative basis, with restrictions neither on the dimensions nor on the base field $\mathbb F$ and study its structure. If we write  $\frak L = \mathbb V \oplus \mathbb W$ with $\mathbb V$ and $0 \neq \mathbb W$ linear subspaces, 
we say that a basis  of homogeneous elements $\mathfrak{B} = \{e_i\}_{i \in I}$ of $\mathbb W$ is quasi-multiplicative if given $0 < k < n,$ for $i_1,\dots,i_k \in I$ and $\sigma \in \mathbb S_n$ satisfies $\langle e_{i_1}, \dots, e_{i_k}, \mathbb V, \dots, \mathbb V \rangle_{\sigma} \subset \mathbb{F}e_{j_{\sigma}}$ for some $j_{\sigma} \in I;$ the product of elements of the basis $\langle e_{i_1}, \dots, e_{i_n} \rangle$ belongs to $\mathbb{F}e_j$ for some $j \in I$ or to $\mathbb V$, and a similar condition is verified for the product $\langle \mathbb V, \dots,\mathbb V \rangle$. We state that if $\frak L$ admits a quasi-multiplicative basis then it decomposes as 
$\mathfrak{L} ={\mathcal U} \oplus (\sum\limits {\frak J}_{k})$ with  any ${\frak J}_k$ a well described color gLt-ideal of $\frak L$ admitting also a quasi-multiplicative basis, and ${\mathcal U}$ a linear subspace of $\mathbb V$.
 Also the minimality of $\frak L$ is characterized in terms of the connections and it is shown that the above direct sum is by means of the family of its minimal color gLt-ideals, admitting each one a $\mu$-quasi-multiplicative basis inherited by the one of $\frak L$.
 
 \
 

\


\medskip
\end{abstract}

\maketitle

\

\ 

\section{Introduction}

 The concept of multiplicative bases appears in a natural way in the study of different physical problems. In fact, one may expect that there are problems which are naturally and more simply formulated exploiting multiplicative bases. 
 There are many works about a study of algebras with a multiplicative basis\cite{m7,kmod,m5,vitya00,vitya03,bgs11,bgs15,Yo_Arb_Alg,Yo_modules,Yo_Arb_Triple,Yo_n_algebras,m1,m6,m3,poj1,poj3,bai1,bai2}.
 
 The interest in the study of  associative algebras admitting a quasi-multiplicative basis also comes from the viewpoint of the theory of infinite dimensional Lie algebras. Since we can recover many classes of Lie algebras from an associative algebra with involution, (see for instance  \cite[Section 6]{Kochetov}), it is interesting to know the structure of the initial associative algebra to understand the one of the Lie algebra. In the framework of infinite dimensional Lie algebras we find many classes of Lie algebras which admits, in  a  natural way, a quasi-multiplicative basis. For instance, we have the semisimple separable $L^*$-algebras, the semisimple locally finite split Lie algebras over a  field of characteristic zero and the graded Lie algebras considered in \cite[Section 3]{Yo4}. Taking into account these comments, it seems natural to study associative and non-associative algebras with a quasi-multiplicative basis to a better understanding of these classes of Lie algebras. 
 The first attempt of a study of algebras with a quasi-multiplicative basis was given in \cite{cuasi}. 

 This interest leads us  in a natural way to study Lie algebras admitting quasi-multiplicative basis. But instead of study this class of algebras, we are going to extend our framework in two different ways, on the one hand we will consider the widest $n$-ary extension by dealing with generalized Lie-type  algebras. On the other hand we will consider the colored version of this category of algebras, that is, the generalized Lie-type color algebras. Hence, our aim is to study  generalized Lie-type color algebras admitting quasi-multiplicative bases.

The paper is organized as follows. 
In the section \ref{decomposition}  we introduce relations techniques on the set of indexes $I$ of the quasi-multiplicative basis  so as to become a powerful  tool for the study of this class of algebras. By making use of these techniques we show that any  generalized Lie-type color algebra $\frak{L}$ admitting a  quasi-multiplicative basis is of the form $\frak{L} ={\mathcal U} \oplus (\sum\limits {\frak J}_{k})$
with  any ${\frak J}_k$ a well described color gLt-ideal of $\frak L$ admitting also a quasi-multiplicative basis, and ${\mathcal U}$ a linear subspace of $\mathbb V$. 
In the section \ref{minimal} the minimality of $\frak L$ is characterized in terms of the quasi-multiplicative basis and it is shown that, under mild conditions, the above decomposition of $\frak L$ is actually the direct sum of the family of its minimal color gLt-ideals. 

\medskip

\section{Previous definitions}


\subsection{Generalized Lie-type algebras}

Let us introduce now the notion of generalized Lie-type algebras which extends some well-known classes of algebras.


\begin{definition}\rm
An $n$-ary algebra $(\frak L, \langle \cdot, \overset{n)}{\dots}, \cdot \rangle)$ is called a {\it generalized Lie-type algebra} (or shortly {\it gLt-algebra}) if it satisfies  the following $n$ identities:
\begin{eqnarray}\label{gLt_condition}
\begin{split}
& \langle y_1,  \ldots, \underbrace{\langle x_1, \dots, x_n \rangle}_{\mbox{pos } k} , \dots, y_{n-1} \rangle =\\
& \quad \quad \quad \sum\limits_{{\tiny
\begin{array}{l}
  1 \leq i,j \leq n \\
  \sigma_1 \in \mathbb S_n\\
  \sigma_2 \in \mathbb S_{n-1} \\
\end{array}}}
\alpha_{i,j,k}^{\sigma_1,\sigma_2} \langle x_{\sigma_1(1)}, \dots, x_{\sigma_1(i-1)}, \langle y_{\sigma_2(1)}, \dots, \underbrace{x_{\sigma_1(i)}}_{\mbox{pos } j}, \dots, y_{\sigma_2(n-1)} \rangle, x_{\sigma_1(i+1)}, \dots, x_{\sigma_1(n)} \rangle,
\end{split}
\end{eqnarray}
for $k=1, \ldots, n,$  $x_1,\dots,x_n,y_1,\dots,y_{n-1} \in L$, being $\alpha_{i,j,k}^{\sigma_1,\sigma_2} \in \mathbb{F}$, and where \mbox{pos} $j$ means that the element $x_{\sigma_1(i)}$ is in the position $j$ in the inside $n$-product.
\end{definition}

\noindent Observe that depending of the values of $\alpha_{i,j,k}^{\sigma_1,\sigma_2}$ we obtain several binary algebras:
\begin{enumerate}
\item[$\bullet$] Lie algebras, Leibniz algebras, Novikov algebras,
associative algebras,  alternative algebras, bicommutative algebras, commutative pre-Lie algebras,  etc.;
\end{enumerate}
and several $n$-ary algebras:
\begin{enumerate}
\item[$\bullet$] $n$-Lie (Filippov) algebras, commutative Leibniz $n$-ary algebras,
totally associative-commutative $n$-ary algebras, etc.
\end{enumerate}

\subsection{Color $\Omega$-algebras}
In this subsection we discuss about color $n$-ary algebras and color $\Omega$-algebras.
In the end of the subsection we give some defintions of classical color algebras.

\begin{definition}\rm
Let $\mathbb{G}$ be an abelian group. A {\it graded $n$-ary algebra} $(\frak L, \langle \cdot, \overset{n)}{\dots}, \cdot \rangle)$ is a $\mathbb{G}$-graded vector space $\frak L = \oplus_{g \in \mathbb{G}}\frak L_g$ provided with a graded $n$-linear map 
$\langle \cdot,\overset{n)}{\dots},\cdot \rangle : \frak L \times \cdots \times \frak L \to \frak  L$ satisfying
$$\langle \frak L_{g_1}, \dots, \frak L_{g_n} \rangle
\subset \frak L_{g_1 + \cdots + g_n},$$
\end{definition}
\noindent for $g_1,\dots,g_n \in \mathbb{G}.$

\begin{definition}\rm
Let $\mathbb{F}$ be a field and $\mathbb{G}$ an abelian group. A map $\epsilon : \mathbb{G} \times \mathbb{G} \to \mathbb{F} \setminus \{0\}$ is called a {\it bicharacter on} $\mathbb{G}$ if it satisfies:
\begin{enumerate}
\item[1.] $\epsilon(k, g+h) = \epsilon(k,g)\epsilon(k, h)$,
\item[2.] $\epsilon(g+h,k) = \epsilon(g,k)\epsilon(h,k)$,
\item[3.] $\epsilon(g,h)\epsilon(h,g) = 1$.
\end{enumerate}
for all $g, h, k \in \mathbb{G}.$
\end{definition}

Let $\frak L = \oplus_{g \in \mathbb{G}}\frak L_g$ be a graded  $n$-ary algebra. An element $x$ is called a {\it homogeneous element of degree} $g$ if $x \in\frak  L_g$ and denoted by $\deg(x) = g$. From now on, unless stated otherwise, we assume that all elements are homogeneous. Let $\epsilon$ be a bicharacter of $\mathbb{G}.$ Given two homogeneous elements $x,y \in \frak L$ we set $\epsilon(x,y) := \epsilon(\deg(x),\deg(y))$.
Now we recall the notion of color $n$-ary $\Omega$-algebra for an arbitrary family of polynomial identities $\Omega$ (see \cite{color n-ary,beites18} for more details).

\begin{definition}\rm
For a (possible  $n$-ary) multilinear polynomial $f(x_1,\dots, x_n)$ we fix the order of indexes $\{i_1,\dots,i_n\}$ of one non-associative word
$\langle x_{i_1},\dots,x_{i_n}\rangle _{\beta}$ from the polynomial $f$. Here,
$$f = \sum_{\beta,\sigma \in S_n} \alpha_{\sigma,\beta} \langle x_{\sigma(i_1)},\dots,x_{\sigma(i_n)} \rangle_{\beta},$$ where $\mathbb{S}_n$ is the permutation group of $n$ elements and $\beta$ is an arrangement of brackets in the non-associative word. For the shift $\mu_i : \{j_1,\dots,j_n\} \mapsto \{j_1,\dots,j_{i+1},j_i,\dots,j_n\}$ we define the element $\epsilon(x_{j_i},x_{j_{i+1}}) \in {\mathbb F} \setminus \{0\}.$ Now, for arbitrary non-associative word $\langle x_{\sigma(i_1)},\dots,x_{\sigma(i_n)} \rangle_{\beta}$ its order of indexes is a composition of suitable shifts $\mu_i$, and for this word we set $\epsilon_{\sigma}$ defined as the product of corresponding $\epsilon(x_{j_i}, x_{j_{i+1}})$. Now, for the multilinear polynomial $f$, we define the {\it color multilinear polynomial} $$f_{co} = \sum_{\beta,\sigma \in S_n}\alpha_{\sigma,\beta}\epsilon_{\sigma}\langle x_{\sigma(i_1)},\dots,x_{\sigma(i_n)}\rangle_{\beta}.$$

\noindent Let $\Omega = \{f_i\}$ be a family of $n$-ary multilinear polynomials. We say that a $n$-{\it ary} algebra $\frak L$ is a $\Omega$-{\it algebra} if it satisfies the family of polynomial identities $\Omega = \{f_i\}.$ Also an $n$-{\it ary}  {\it color}  $\Omega$-{\it algebra} is an $n$-ary color algebra $\frak L$ satisfying the family of color multilinear polynomials $\Omega_{co} = \{(f_i)_{co}\}$.
\end{definition}

Some examples of $n$-ary color algebras
are 
Lie and Jordan superalgebras \cite{ivan09,ivan12},
Lie color algebras \cite{mikh85},
Leibniz color algebras \cite{dzhuma},
Filippov ($n$-Lie) superalgebras \cite{poj08patricia,ck10,poj03,poj09,poj08}
and $3$-Lie color algebras \cite{ZT15}. Let us give  definitions of some color algebras.

\begin{definition}\rm
\label{colorliebin}
{A Leibniz color algebra $(\frak L,[\cdot,\cdot], \epsilon)$ is a $\mathbb{G}$-graded vector space $\frak L=\bigoplus_{g \in \mathbb{G}} \frak L_g$ with a bicharacter $\epsilon$, an even bilinear map $[\cdot,\cdot]: \frak L \times \frak L \rightarrow\frak  L$  satisfying
$$[x,[y,z]] = [[x,y],z]+ \epsilon(x,y)[y,[x,z]].$$
}\end{definition}

\begin{definition}\rm
An  $n$-Lie color algebra $(\frak L, [\cdot, \ldots, \cdot], \epsilon)$ is a $\mathbb{G}$-graded vector space  $\frak L=\bigoplus_{g \in \mathbb{G}} \frak L_g$
with an $n$-linear map $[\cdot,\ldots, \cdot]: \frak L \times \ldots \times \frak L \rightarrow \frak L$ satisfying
$$[x_1, \ldots, x_i, x_{i+1}, \ldots, x_n]=- \epsilon( x_i, x_{i+1}) [x_1, \ldots, x_{i+1}, x_i, \ldots, x_n],$$
$$[x_1, \ldots, x_{n-1}, [y_1, \ldots, y_n]] =\sum_{i=1}^n \epsilon( X_{n-1}, Y_{i-1}) [y_1, \ldots,  y_{i-1},[x_1,\ldots, x_{n-1},y_i], y_{i+1},\ldots, y_n],$$
where $X_{i-1}=\sum\limits_{k=1}^{i-1} x_k, \ Y_{i-1}=\sum\limits_{k=1}^{i-1} y_k.$
\end{definition}

\begin{definition}\rm
 For any $\sigma$ from the permutation group of $n$ elements $ \mathbb{S}_n,$ we use the notation 
 $$\langle x_1,\dots,x_j,\dots,x_n \rangle_{\sigma} = \langle x_{\sigma(1)},\dots,x_{\sigma(j)},\dots,x_{\sigma(j)} \rangle.$$
\end{definition}

\subsection{Multiplicative and quasi-multiplicative basis}
Firstly we establish the natural definition of multiplicative basis of graded $n$-ary algebras.

\begin{definition}\label{11}\rm
A basis of homogeneous elements $\mathfrak{B} = \{e_i\}_{i \in I}$ of a graded $n$-ary algebra $\frak L$ is {\it multiplicative} if for any $i_1,\dots,i_n \in I$ we have $\langle e_{i_1},\dots,e_{i_n} \rangle \in \mathbb{F}e_j$ for some $j \in I$.
\end{definition}

\noindent In particular, this definition extends the one considered in \cite{Yo_Arb_Triple,Yo_n_algebras}. Also this definition is more general than the usual one in the literature \cite{m1,m6,m3}. In fact, in these references, a basis ${\mathfrak B} = \{e_i\}_{i \in I}$ is {\it multiplicative} if for any $i,j \in I$ we have either $e_ie_j=0$ or $0 \neq e_ie_j = e_k$ for some $k \in I$.

We wish to go a further step by introducing a more general concept than the one of  multiplicative basis as follows.

\begin{definition}\label{2}\rm
A graded $n$-ary algebra $\frak L$ admits a {\it quasi-multiplicative basis} if $\frak L = \mathbb V \oplus \mathbb W$ with $\mathbb V$ and $0 \neq \mathbb W$ graded linear subspaces in such a way that there exists a basis of homogeneous elements $\mathfrak{B} = \{e_i\}_{i \in I}$ of $\mathbb W$ satisfying:
\begin{enumerate}
\item[1.] For $i_1,\dots,i_n \in I$ we have either $\langle e_{i_1}, \dots, e_{i_n} \rangle \in \mathbb{F}e_j$ for some $j \in I$ or $\langle e_{i_1},\dots, e_{i_n} \rangle \in \mathbb V$.
\item[2.] Given $0 < k < n,$ for $i_1,\dots,i_k \in I$ and $\sigma \in \mathbb S_n$ we have $\langle e_{i_1}, \dots, e_{i_k}, \mathbb V, \dots, \mathbb V \rangle_{\sigma} \subset \mathbb{F}e_{j_{\sigma}}$ for some $j_{\sigma} \in I$.
\item[3.] We have either $\langle \mathbb V, \dots, \mathbb V \rangle \subset \mathbb{F}e_j$ for some $j \in I$ or $\langle \mathbb V, \dots,\mathbb V \rangle \subset\mathbb V$.
\end{enumerate}
\end{definition}

\noindent Observe that in item 2. we only consider $0 < k < n$ because $k=n$ is contemplated in item 1. and $k=0$ in item 3. We also note that if the linear subspace $\mathbb V$ is trivial or $1$-dimensional this definition agrees with the one of multiplicative basis.

We note that a different concept of quasi-multiplicative basis to the one given in Definition \ref{2} can be found, in a context of categories theory, in the reference \cite{m7}.

Examples of gLt-algebras admitting a quasi-multiplicative basis are any $n$-ary algebras admitting a multiplicative basis (case $\mathbb V = \{0\}$).  We also have that any finite-dimensional associative algebra $A$ of finite representation type (that is, there are only finitely many isomorphism classes of indecomposable finite-dimensional $A$-modules) has also a multiplicative basis \cite{m5,m3}.  Multiplicative bases are also well-related to Gr${\rm \ddot{o}}$bner basis. In fact, it is well-known that an algebra with a multiplicative basis have a Gr${\rm \ddot{o}}$bner basis theory if there is an admissible order on the basis \cite{m6}.


\section{Decomposition as direct sum of ideals}\label{decomposition}

In what follows $\frak  L = \mathbb V \oplus\mathbb W$ denotes a color gLt-algebra admitting a quasi-multiplicative basis of homogeneous elements ${\mathfrak B}=\{e_i\}_{i \in I}$ of $W \neq 0$. We begin this section by developing connection techniques among the elements in the set of indexes $I$ as a main tool in our study.

Consider $\mathbb V$ an external element to $I$ and define the set $${\frak I} := I \dot{\cup} \{v\}.$$ The element $\mathbb V$ gives us information about the behavior of the linear subspace $\mathbb V$ with respect to the elements in the basis ${\mathfrak B}$. For each $j \in {\frak I}$, a new assistant variable $\overline{j} \notin {\frak I}$ is introduced and we consider the set $\overline{I} := \{\overline{i} : i \in I\},$ so $$\overline{{\frak I}} := \{ \overline{j} : j \in {\frak I} \} = \overline{I} \dot{\cup} \{\overline{v}\}$$ consists of all these new symbols. We also denote by ${\mathcal P}(A)$ the power set of a given set $A$.

\medskip

Next, we consider the following operations which recover, in a sense, certain multiplicative relations among the elements in ${\frak I}$. Given $\sigma \in \mathbb S_n$ we define 

\begin{enumerate}
    \item[$\star$] $a_\sigma : {\frak I} \times \overset{n)}{\ldots} \times {\frak I} \to \mathcal P({\frak I})$ such as

\begin{itemize}
\item $a_\sigma(i_1,\dots,i_n) := \left\{\begin{array}{cll}
    \{r\}, & \hbox{if $\; 0 \neq \langle e_{i_{\sigma(1)}},\dots,e_{i_{\sigma(n)}} \rangle \in {\mathbb F}e_r$;} \\
    \{v\}, & \hbox{if $\; 0 \neq \langle e_{i_{\sigma(1)}},\dots,e_{i_{\sigma(n)}} \rangle \in \mathbb V$.} \\
\end{array}\right.$

\medskip

\item For $0 < k < n,$\\ $a_\sigma(i_1,\dots,i_k,v,\dots,v) := \left\{\begin{array}{cl}
    \{r\}, & \hbox{if $\; 0 \neq \langle e_{i_1}, \dots, e_{i_k}, \mathbb V, \dots, \mathbb V \rangle_{\sigma} \subset {\mathbb F}e_r$} \}. \\ 
   \end{array}\right.$

\medskip

\item $a_\sigma(v,\dots,v) := \left\{\begin{array}{cll}
    \{r\}, & \hbox{if $\; 0 \neq \langle \mathbb V,\dots, \mathbb V \rangle \subset {\mathbb F} e_r$;} \\
    \{v\}, & \hbox{if $\; 0 \neq \langle \mathbb V,\dots, \mathbb V \rangle \subset \mathbb V$.} \\
\end{array}\right.$

\medskip

\item $\emptyset$ in the remaining cases.
\end{itemize}

\item[$\star$]  $b_\sigma : \frak{I} \times \overline{\frak{I}} \times \overset{n-1)}{\ldots} \times \overline{\frak{I}} \to \mathcal{P}(\frak{I})$ such as
\begin{itemize}
\item For $1 \leq k \leq n,$\\ $b_\sigma(i,\overline{i}_2,\dots,\overline{i}_k,\overline{v},\dots,\overline{v}) :=$ 
$$\Bigl\{i' \in I : a_\sigma(i',i_2,\dots,i_k,v,\dots,v) = \{i\} \Bigr\} \cup 
\Bigl\{
                                               \begin{array}{ll}
                                                 \{v\}, & \hbox{if $a_\sigma(i_2,\dots,i_k,v,\dots,v) = \{i\}$} \Bigr\}.
                                               \end{array}$$

\medskip



\item $b_\sigma(v,\overline{i}_2,\dots,\overline{i}_n) :=  \{i' \in I: a_{\sigma}(i',i_2,...,i_n)=\{v\}$ \}.

\medskip

\item $b_\sigma(v,\overline{v},\dots,\overline{v}) := \{
                                               \begin{array}{ll}
                                                 \{v\}, & \hbox{if $a_\sigma(v,\dots,v) =\{v\}$}  \}. 
                                               \end{array}$

\medskip

\item $\emptyset$ in the remaining cases.
\end{itemize}

\end{enumerate}

\noindent Then, we consider the operation $$\mu : (\frak{I} \; \dot{\cup} \; \overline{\frak{I}}) \times ((\frak{I} \times \overset{n-1)}{\ldots} \times \frak{I})   \dot{\cup} (\overline{\frak{I}} \times \overset{n-1)}{\ldots} \times \overline{\frak{I}})) \to \mathcal{P}(\frak{I})$$ given by:

\medskip

\begin{itemize}
\item For any $j,j_1,...,j_{n-1} \in \frak{I}$, \hspace{0.1cm} $\mu(j,j_1,\dots,j_{n-2},j_{n-1}) := \bigcup\limits_{\sigma \in \mathbb S_n} a_{\sigma}(j,j_1,\dots,j_{n-2},j_{n-1})$

\medskip

\item For any $j \in \frak{I}$ and $\overline{j}_1,...,\overline{j}_{n-1} \in \overline{\frak{I}}$,  \hspace{0.1cm} $\mu(j,\overline{j}_1,\dots,\overline{j}_{n-2},\overline{j}_{n-1}) := \bigcup\limits_{\sigma \in\mathbb S_n} b_{\sigma}(j,\overline{j}_1,\dots,\overline{j}_{n-2},\overline{j}_{n-1})$

\medskip

\item For any $j,j_1,...,j_{n-1} \in \frak{I}$, \hspace{0.1cm} $\mu(\overline{j},{j}_1,\dots,{j}_{n-1}) := \bigcup\limits_{k\in\{1,...,n-1\}, \sigma \in\mathbb S_{n-1}}  b_{\sigma}(j_k, \overline{j},\overline{j}_1,...,\overline{j}_{k-1},\overline{j}_{k},...,\overline{j}_{n-1})$

\medskip

\item  $\mu(\overline{\frak{I}},\overline{\frak{I}},\dots,\overline{\frak{I}}) := \emptyset.$

\end{itemize}

\noindent From now on, given any $\overline{j} \in \overline{\frak{I}}$ we denote $\overline{(\overline{j})} := j.$ Given also any subset ${J}$ of ${\frak I} \dot{\cup} \overline{{\frak I}}$, we write by $\overline{J} := \{\overline{j} : j \in J\}$ if $J \neq \emptyset$ and $\overline{\emptyset}:= \emptyset$.

\begin{lemma}\label{lema1}
Let $i,j \in I$ and elements $a_2,\dots,a_n$ of $(\frak{I} \times \overset{n-1)}{\ldots} \times \frak{I})   \dot{\cup} (\overline{\frak{I}} \times \overset{n-1)}{\ldots} \times \overline{\frak{I}}).$ It holds that $i \in \mu(j,a_2,\dots,a_n)$ if and only if $j \in\mu(i,\overline{a}_2,\dots,\overline{a}_n)$.
\end{lemma}

\begin{proof}
First let us suppose $i \in \mu(j,a_2,\dots,a_n)$. If $\{a_2,\dots,a_n\} \subset \frak{I}$, then  there exists $\sigma \in \mathbb S_n$ such that $\{i\} = a_{\sigma}(j,a_2,\dots,a_n).$ Then $$j \in b_{\sigma}(i,\overline{a}_2,\dots,\overline{a}_n) \subset \mu(i,\overline{a}_2,\dots,\overline{a}_n).$$

\noindent In the another case, if $\{a_2,\dots,a_n\} \subset \overline{\frak{I}}$, then exists $\sigma \in \mathbb S_n$ such that $i \in b_{\sigma}(j,{a_2},\dots,{a_n})$ and so $$\{j\} = a_{\sigma}(i, \overline{a}_2, \dots,\overline{a}_n) \subset \mu(i,\overline{a}_2,\dots,\overline{a}_n).$$

\noindent To prove the converse we can argue in a similar way.
\end{proof}

\begin{lemma}\label{lema111}
 Let $i \in I$, $j \in \overline{I}$ and elements $a_1,\dots,a_{n-1}$ of $(\frak{I} \times \overset{n-1)}{\ldots} \times \frak{I})   \dot{\cup} (\overline{\frak{I}} \times \overset{n-1)}{\ldots} \times \overline{\frak{I}}).$ It holds that $i \in \mu(j,a_1,\dots,a_{n-1})$ if and only if
$\overline{j} \in \mu(\overline{i}, a_1, \dots, a_{n-1})$.

\end{lemma}

\begin{proof}
Suppose $i \in \mu(j,a_1,\dots,a_{n-1})$.
Since $\mu(\overline{\frak{I}},\overline{\frak{I}},\dots,\overline{\frak{I}}) := \emptyset,$ we just have to consider   the case in which  $\{a_1,\dots,a_{n-1}\} \subset \frak{I}$. Then there exist $k \in \{1,...,n-1\}$ and $\sigma \in\mathbb S_{n-1}$ such that $ i \in  b_{\sigma}(a_k,{j},\overline{a}_1,..,\overline{a}_{k-1},\overline{a}_{k+1},\dots,\overline{a}_{n-1}).$
From here, $$\{a_k\} = a_{\sigma}(i,\overline{j},a_1,\dots,a_{k-1},a_k,\dots,a_{n-1})=a_{\nu}(\overline{j},i,a_1,\dots,a_{k-1},a_k,\dots,a_{n-1})$$ where $\nu= (1,2) \sigma$.  Hence
$ \overline{j} \in  b_{\nu}(a_k,\overline{i},\overline{a}_1,..,\overline{a}_{k-1},\overline{a}_{k+1},\dots,\overline{a}_{n-1}) \subset  \mu(\overline{i}, a_1, \dots, a_{n-1}).$

\noindent The converse is proved similarly.

\end{proof}

\medskip

The map $\mu$ is not adequate for our purposes in the sense that we require to send a subset of $I \dot\cup \overline{I}$ in a subset of the same union. So we need to introduce the following map:
$$\phi: {\mathcal P}(I \dot\cup \overline{I}) \times ((\frak{I} \times \overset{n-1)}{\ldots} \times \frak{I})   \dot{\cup} (\overline{\frak{I}} \times \overset{n-1)}{\ldots} \times \overline{\frak{I}})) \to {\mathcal P}(I \dot\cup \overline{I}),$$ as
\begin{itemize}
\item $\phi(\emptyset, (\frak{I} \times \overset{n-1)}{\ldots} \times \frak{I})   \dot{\cup} (\overline{\frak{I}} \times \overset{n-1)}{\ldots} \times \overline{\frak{I}})) := \emptyset$,

\item For any $\emptyset \neq J \in {\mathcal P} (I \dot\cup \overline{I})$ and $a_2,\dots,a_n \in (\frak{I} \times \overset{n-1)}{\ldots} \times \frak{I})   \dot{\cup} (\overline{\frak{I}} \times \overset{n-1)}{\ldots} \times \overline{\frak{I}})$,
 $$\phi({J}, a_2,\dots,a_n) := \Bigl(\bigl(\bigcup\limits_{j \in J} \mu(j, a_2, \dots, a_n)\bigr) \setminus \{v\} \Bigr) \cup \overline{\Bigl(\bigl(\bigcup\limits_{j \in J} \mu(j, a_2, \dots, a_n)\bigr) \setminus \{v\} \Bigr)}.$$
\end{itemize}

\medskip

\noindent Note that for any $J \in {\mathcal P}(I \dot\cup \overline{I})$ and $a_2,\dots,a_n \in (\frak{I} \times \overset{n-1)}{\ldots} \times \frak{I})   \dot{\cup} (\overline{\frak{I}} \times \overset{n-1)}{\ldots} \times \overline{\frak{I}})$ we have that
\begin{equation}\label{bar}
\phi(J, a_2,\dots,a_n) = \overline{\phi(J, a_2,\dots,a_n)}
\end{equation}
and
\begin{equation}\label{barro}
\phi(J, a_2,\dots,a_n) \cap I = \Bigl(\bigcup\limits_{j \in J} \mu(j, a_2, \dots, a_n) \Bigr) \setminus \{v\}.
\end{equation}

\begin{lemma}\label{1}
Consider $J \in {\mathcal P} (I\dot{\cup} \overline{I})$ such that $J = \overline{J},$ $a_2,\dots,a_n \in {\frak I} \dot{\cup} \overline{{\frak I}}$ and $i \in I$. The following statements are equivalent:
\begin{enumerate}
\item[1.] $i \in \phi(J,a_2,\dots,a_n);$
\item[2.] either $\phi(\{i\}, \overline{a}_2, \dots, \overline{a}_n) \cap J \cap I \neq \emptyset$ or  $\phi(\{\overline{i}\}, a_2, \dots,a_n) \cap J \cap I \neq  \emptyset$. 
\end{enumerate}
\end{lemma}

\begin{proof}
It is straightforward to verify that for any $i \in I$ and $a_2,\dots,a_n \in {\frak I} \dot{\cup} \overline{{\frak I}}$ we have that $$i \in \mu(j,a_2,\dots,a_n)$$ for some $j \in I$ if and only if $j \in \mu(i, \overline{a}_2, \dots, \overline{a}_n)$ (see Lemma \ref{lema1}),  while $i \in \mu(j, a_2,\dots, a_n)$ for some $j \in \overline{I}$ if and only if $\overline{j} \in \mu(\overline{i}, a_2,\dots,a_n).$  Since these facts together with Equations \eqref{bar} and \eqref{barro} we conclude the result.
\end{proof}

For an easier comprenhesion we firstly present a shorter notation. Given $m$ a natural number, we denote $X_m := (a_{m,2},\dots,a_{m,n}) \in \frak{I} \;\dot\cup\; \overline{\frak{I}} \times \overset{n-1)}{\dots} \times \frak{I} \;\dot\cup\; \overline{\frak{I}}$. Let us also denote $$\overline{X}_m := (\overline{a}_{m,2},\dots,\overline{a}_{m,n}).$$

\noindent Additionally, for $t \geq 1,$ by $\{X_1,X_2,\dots,X_t\}$ we mean the set of elements $$\{a_{1,2},\dots,a_{1,n},a_{2,2},\dots,a_{2,n},\dots,a_{t,2} \dots, a_{t,n}\}.$$

\noindent Finally, for $\mathfrak{A} \in \mathcal{P}(I \dot\cup \overline{I})$ we denote $\phi(\mathfrak{A},X_m) := \phi(\mathfrak{A},a_{m,2},\dots,a_{m,n}).$

\begin{definition}\label{defco}\rm
Let $i$ and $j$ be distinct elements in $I$. We say that $i$ is {\it connected} to $j$ if there exists a subset $\{X_1,\dots,X_t\} \subset \frak{I}\;\dot\cup\; \overline{\frak{I}},$ for certain $t \geq 1$, such that the following conditions hold:

\begin{enumerate}
\item [{\rm 1.}] $\widetilde{i} \in \{i,\overline{i}\},$

\bigskip

\item [{\rm 2.}] $\phi(\{\widetilde{i}\},X_1) \neq \emptyset$,\\
$\phi(\phi(\{\widetilde{i}\},X_1),X_2) \neq\emptyset$,\\
$\hspace*{2cm} \vdots$\\
$\phi(\phi(\dots\phi(\{\widetilde{i}\},X_1),\dots), X_{t-1}) \neq\emptyset$.

\bigskip

\item [{\rm 3.}] $j \in \phi(\phi(\dots\phi(\{\widetilde{i}\},X_1),\dots),X_t).$
\end{enumerate}

\bigskip

\noindent The subset $\{X_1,\dots,X_t\}$ is a {\it connection} from $i$ to $j$ and by convention $i$ is connected to itself.
\end{definition}

\noindent Our aim is to show that the connection relation is of equivalence. Previously we check the symmetric property.








\begin{proposition}\label{proequiv}
The relation $\sim$ in $I$, defined by $i \sim j$ if and only if $i$ is connected to $j$,  is an equivalence relation.
\end{proposition}

\begin{proof}
The reflexive character of $\sim$ is given by the Definition \ref{defco}. Let us see the symmetric character of $\sim$: If $i \sim j$ then there exists a connection $$\{X_1,\dots,X_t\} \subset {\frak I} \dot{\cup} \overline{{\frak I}}$$ from $i$ to $j$ satisfying conditions in Definition \ref{defco}. In case $t=1$ we have $j \in \phi(\{\widetilde{i}\}, X_1).$ If $\widetilde{i} = i \in I$ then $i \in \phi(\{j\},\overline{X}_1)$. If $\widetilde{i} = \overline{i} \in \overline{I}$  then $i \in \phi(\{\overline{j}\},X_1)$.  So $i \in \phi(\{\widetilde{j}\},\widetilde{X}_1)$ with $(\widetilde{j},\widetilde{X}_1) \in \{(j,\overline{X}_1),(\overline{j},X_1)\},$ that is, $\{\widetilde{X}_1\}$ is a connection from $j$ to $i$.

Suppose $t \geq 2$ and let us show that we can find a set $$\{\widetilde{X}_t, \dots, \widetilde{X}_1\} \subset {\frak I} \dot{\cup} \overline{{\frak I}},$$ where $\widetilde{X}_m \in \{X_m, \overline{X}_m\}$, for $1 \leq m \leq t$, which gives rise to a connection from $j$ to $i$. Indeed, Equation \eqref{bar} shows $$\phi (\dots(\phi(\{\widetilde{i}\}, X_1), \dots),X_{t-1}) = \overline{\phi (\dots (\phi(\{\widetilde{i}\}, X_1), \dots ),X_{t-1})}$$ and so by taking $J := \phi(\dots(\phi(\{\widetilde{i}\}, X_1), \dots ),X_{t-1})$ we have $J \in \mathcal{P}(I \dot\cup \overline{I})$ and $J = \overline{J},$ so we can apply Lemma \ref{1} to the expression $$j \in \phi(\phi (\dots (\phi(\{\widetilde{i}\}, X_1), \dots ),X_{t-1}), X_t)$$
to get  that either $$ \phi(\{j\},\overline{X}_t) \cap \phi(\dots(\phi(\{\widetilde{i}\}, X_1), \dots  ),X_{t-1}) \cap I \neq \emptyset$$ or
 $$ \phi(\{\overline{j}\},X_t) \cap \phi(\dots (\phi(\{\widetilde{i}\}, X_1), \dots ),X_{t-1}) \cap I \neq \emptyset$$  and so  $$\phi(\{\widetilde{j}\},\widetilde{X}_t)  \neq \emptyset$$ with $(\widetilde{j},\widetilde{X}_t) \in
\{(j,\overline{X}_t),(\overline{j},X_t)\}.$

By taking $$k \in \phi(\{\widetilde{j}\},\widetilde{X}_t) \cap \phi(\dots (\phi(\{\widetilde{i}\}, X_1), \dots), X_{t-1}) \cap I,$$ Equation \eqref{bar} and Lemma \ref{1}, the fact $k \in \phi(\cdots (\phi(\{\widetilde{i}\}, X_1), \dots), X_{t-1})$ and $k \in \phi(\{\widetilde{j}\},\widetilde{X}_t)$ imply now either
$$\phi(\phi(\{\widetilde{j}\},\widetilde{X}_t),\overline{X}_{t-1}) \cap \phi(\dots \phi(\phi(\{i\}, X_1), \dots),X_{t-2}) \cap I  \neq \emptyset$$ or  $$\phi( \phi(\{\widetilde{j}\},\widetilde{X}_t),X_{t-1}) \cap \phi(\dots\phi(\phi(\{\widetilde{i}\}, X_1), \dots),X_{t-2}) \cap I \neq \emptyset$$  and consequently $$\phi(\phi( \{\widetilde{j}\}, \widetilde{X}_t), \widetilde{X}_{t-1}) \neq \emptyset$$ for some $\widetilde{X}_{t-1} \in \{X_{t-1}, \overline{X}_{t-1}\}$.

By iterating this process we get $$\phi(\phi(\dots (\phi(\{\widetilde{j}\} , \widetilde{X}_t),\dots),\widetilde{X}_{t-m+1}), \widetilde{X}_{t-m}) \cap$$ $$\phi(\phi(\dots (\phi(\{\widetilde{i}\}, X_1), \dots),X_{t-m-2}), X_{t-m-1}) \cap I \neq \emptyset$$ for $0 \leq m \leq t-2.$ In particular, we have for the case $m=t-2$ that $$\phi(\phi(\cdots (\phi(\{\widetilde{j}\} , \widetilde{X}_t), \dots),\widetilde{X}_3), \widetilde{X}_2) \cap \phi(\{\widetilde{i}\}, X_1) \cap I \neq \emptyset.$$ Since either $\widetilde{i} = i$ or $\widetilde{i} = \overline{i}$, if we write $J := \phi(\phi(\dots (\phi(\{\widetilde{j}\} , \widetilde{X}_t), \dots),\widetilde{X}_3), \widetilde{X}_2)$, the previous equation allows us to assert that either $\phi(\{i\},X_1)\cap J \cap I \neq \emptyset$ or  $\phi(\{\overline{i}\},X_1) \cap J \cap I \neq \emptyset$  with $i \in I$. Hence  Lemma \ref{1} applies to get $$i \in \phi(\phi(\dots(\phi(\{
\widetilde{j}\},\widetilde{X}_t), \dots), \widetilde{X}_2), \widetilde{X}_1)$$ for some $\widetilde{X}_1 \in \{X_1, \overline{X}_1\}$ and conclude $\sim$ is symmetric.

Finally, let us verify the transitive character of $\sim$. Suppose $i \sim j$ and $j \sim k$, and write $\{X_1,\dots,X_t\}$  for a connection from $i$ to $j$ and $\{Y_1,\dots, Y_s\}$ for a connection from $j$ to $k$. If $i=j$ so $\{Y_1,\dots, Y_s\}$ is a connection from $i$ to $k$. If $k=j$ thus $\{X_1,\dots,X_t\}$ is a connection from $i$ to $k$. Finally, if $t
\geq 1$ and $s \geq 1$, taking into account Equation \eqref{bar} we easily have that $\{X_1,\dots,X_t,Y_1,\dots,Y_s\}$ is a connection from $i$ to $k$.
We have shown that the connection relation is an equivalence relation.
\end{proof}

\noindent By the above proposition we can consider the quotient set $$I / \sim = \{[i] : i \in I\},$$ becoming $[i]$ the set of elements in $I$ which are connected to $i$.

\begin{definition}\rm
A {\it color  gLt-subalgebra} of $\frak L$ is a $\mathbb G$-graded subspace $\frak S$ of $\frak L$ verifying 
$\langle \frak  S, \dots, \frak S \rangle \subset \frak S$. A $\mathbb G$-graded subspace $\mathcal{I} \subset\frak  L$ is a {\it color gLt-ideal} of $\frak L$ if $\langle \mathcal{I}, \frak L, \dots, \frak L \rangle_{\sigma} \subset \mathcal{I},$ for any $\sigma \in \mathbb S_n$.
\end{definition}

Our next goal in this section is to associate an $n$-ary ideal ${\frak J}_{[i]}$ of $\frak L$ to any $[i] \in I/\sim$. Fix $i \in I$, we start by defining the sets
\begin{eqnarray*}
\begin{split}
\mathbb V_{[i]} &:= \Bigl(\sum\limits_{i_1,\dots,i_n \in [i]}{\mathbb F} \langle e_{i_1},\dots,e_{i_n} \rangle \Bigr) \cap \mathbb V \subset \mathbb V,\\
\mathbb W_{[i]} &:= \bigoplus_{j \in [i]}{\mathbb F} e_j \subset \mathbb W
\end{split}
\end{eqnarray*}
Finally, we denote by ${\frak J}_{[i]}$ the direct sum of the two subspaces above, that is, $${\frak J}_{[i]}:= \mathbb V_{[i]} \oplus \mathbb W_{[i]}.$$

\begin{definition}\rm
Let $\frak L = \mathbb V \oplus\mathbb  W$ be a graded $n$-ary algebra admitting a quasi-multiplicative basis ${\mathfrak B} = \{e_i\}_{i \in I}$ of $ \mathbb W \neq 0$. It is said that a $n$-ary graded subalgebra $\frak S$ of $\frak L$ has a quasi-multiplicative basis {\it inherited} by the one of $\frak L$ if $\frak S = \mathbb V_{\frak S} \oplus \mathbb W_{\frak S}$ with $\mathbb V_{\frak  S}$ a graded linear subspace of $\mathbb V$, and $0 \neq \mathbb  W_{\frak S}$ a graded linear subspace of $\mathbb W$ admitting ${\mathfrak B}' \subset {\mathfrak B}$ as a basis.
\end{definition}

\begin{proposition}\label{teo1}
For any $i \in I$, the linear subspace ${\frak J}_{[i]}$ is a color gLt-ideal of $\frak L$ admitting a quasi-multiplicative basis inherited by the one of $\frak L$.
\end{proposition}

\begin{proof}
Given $\sigma \in \mathbb S_n,$ we can write \begin{equation}\label{s1}
\langle {\frak J}_{[i]}, \frak L, \dots, \frak L \rangle_{\sigma} = \langle \mathbb V_{[i]} \oplus \mathbb W_{[i]}, \Bigl(\mathbb V \oplus (\bigoplus_{r \in I}{\mathbb F}e_r)\Bigr), \dots, \Bigl(\mathbb V \oplus (\bigoplus_{s \in I}{\mathbb F}e_s)\Bigr) \rangle_{\sigma}.
\end{equation}

In case $\langle e_j, e_{i_2},\dots,e_{i_n} \rangle_{\sigma} \neq 0$ for some $j \in [i]$ and $i_2,\dots,i_n \in I$, we have that either $0 \neq \langle e_j, e_{i_2},\dots,e_{i_n} \rangle_{\sigma} \in {\mathbb F} e_l$ with $l \in I$ or $0 \neq \langle e_j, e_{i_2},\dots,e_{i_n} \rangle_{\sigma} \in\mathbb V$. In the first case
the connection $\{i_2, \dots, i_n\}$ gives us $j \sim l$, so $l \in [i]$ and then $\langle e_j, e_{i_2},\dots,e_{i_n} \rangle_{\sigma} \in\mathbb W_{[i]}$.   In the second case, for all $2 \leq k \leq n$ we get $i_k \in b_{\tau}(v,\overline{j},\overline{i}_2,\dots,\overline{i}_{k-1},\overline{i}_{k+1},\dots,\overline{i}_n)$ for some $\tau \in\mathbb S_n$, and so $i_k \in \mu (\overline{j},v,{i}_2,\dots,{i}_{k-1},{i}_{k+1},\dots,{i}_n).$ Hence the set $\{v,{i}_2, \dots, {i}_{k-1},{i}_{k+1},\dots,{i}_n\}$ is a connection from $j$ to $i_k$ and so $i_k \in [i]$ for $2 \leq k \leq n$. Therefore $\langle e_j, e_{i_2},\dots,e_{i_n} \rangle_{\sigma} \in V_{[i]}$. Hence we get
\begin{equation}\label{s2}
\langle \mathbb W_{[i]}, (\bigoplus\limits_{r \in I}{\mathbb F} e_r), \dots, (\bigoplus\limits_{s \in I}{\mathbb F} e_s) \rangle_{\sigma} \subset {\frak J}_{[i]}.
\end{equation}

For some $j \in [i],$ if we have $0 \neq \langle e_j, \mathbb V, \dots, \mathbb V, e_{i_2},\dots,e_{i_k}\rangle_{\sigma} \subset \mathbb{F}e_l$ for certain $l \in I$ and where $k\geq 2$, then $l \in \mu(j,v,\dots,v, \dots, i_2, \dots, i_k)$. 
So $\{v,\dots,v,\dots, i_2, \dots, i_k\}$ is a connection from $j$ to $\frak L$  then $l \in [i]$. From here we have (taking also into account Equation (\ref{s2})) that:
\begin{equation}\label{s3}
\langle \mathbb W_{[i]},\frak L,\dots,\frak L\rangle_{\sigma} \subset \mathbb W_{[i]}.
\end{equation}

Suppose there exist $i_1,\dots,i_n \in [i]$ with $0 \neq \langle e_{i_1},\dots, e_{i_n} \rangle_{\sigma} \in \mathbb V$, that is, $v \in \mu(i_1,\dots,i_n)$, in such a way that $0 \neq \langle \langle e_{i_1},\dots, e_{i_n} \rangle_{\sigma}, e_{j_1},\dots,e_{j_{n-1}} \rangle_{\tau} \in {\mathbb F} e_m$ for some $j_1,\dots,j_{n-1} \in I$. By Equation \eqref{gLt_condition} we get that
$$0 \neq \langle e_{i_{\sigma_1(1)}}, \dots, e_{i_{\sigma_1(k-1)}}, \langle e_{j_{\sigma_2(1)}}, \dots, \underbrace{e_{i_{\sigma_1(k)}}}_{\mbox{pos } h}, \dots, e_{j_{\sigma_2(n-1)}} \rangle_{\tau_1}, e_{i_{\sigma_1(k+1)}}, \dots, e_{i_{\sigma_1(n)}} \rangle_{\tau_2} \in {\mathbb F} e_m$$ for certain $1 \leq h,k \leq n$, $\sigma_1 \in \mathbb S_n, \sigma_2 \in \mathbb S_{n-1}$. The connection $\{i_{\sigma_1(2)}, \dots,i_{\sigma_1(k-1)}, i',i_{\sigma_1(k+1)}, \dots i_{\sigma_1(n)}\}$ where either $i' \in I$ if $0 \neq \langle e_{j_{\sigma_2(1)}}, \dots, e_{i_{\sigma_1(k)}}, \dots, e_{j_{\sigma_2(n-1)}} \rangle_{\tau_1} \in {\mathbb F} e_{i'}$ or $i'= v$ when $0 \neq \langle e_{j_{\sigma_2(1)}}, \dots, e_{i_{\sigma_1(k)}}, \dots, e_{j_{\sigma_2(n-1)}} \rangle_{\tau_2} \in\mathbb V,$ gives us $i_{\sigma_1(1)}$ is connected to $m$ and so $m \in [i].$ From here
\begin{equation}\label{s4}
\langle\mathbb V_{[i]}, (\bigoplus_{r \in I}{\mathbb F}e_r) ,\dots, (\bigoplus_{s \in I}{\mathbb F}e_s) \rangle_{\sigma} \subset\mathbb W_{[i]}.
\end{equation}

Now, suppose there exist $i_1,\dots,i_n \in [i]$ with $v \in \mu(i_1,\dots,i_n)$ satisfying $0 \neq \langle \langle e_{i_1},\dots,e_{i_n}\rangle_{\sigma},\mathbb V, \dots, \mathbb V \rangle_{\tau}$. We have two possibilities, in the first one $0 \neq \langle \langle e_{i_1},\dots,e_{i_n} \rangle_{\sigma}, \mathbb V, \dots, \mathbb V \rangle_{\tau} \subset {\mathbb F}e_k$ and we have as above that $k \in [i]$. In the second one 
$0 \neq \langle \langle e_{i_1}, \dots, e_{i_n} \rangle_{\sigma}, \mathbb V, \dots,\mathbb V \rangle_{\tau} \subset \mathbb V$ and we get by Equation \eqref{gLt_condition} $$0 \neq \langle e_{i_{\sigma_1(1)}}, \dots, e_{i_{\sigma_1(k-1)}}, \langle \mathbb V, \dots, e_{i_{\sigma_1(k)}}, \dots, \mathbb V \rangle_{\tau_1}, e_{i_{\sigma_1(k+1)}}, \dots, e_{i_{\sigma_1(n)}} \rangle_{\tau_2} \subset \mathbb V$$ being then $\mu(i_{\sigma_1(k)},v,\dots,v) = \{r\}$, for $r \in I$, with $\{v,\dots,v\}$ a connection from $i_{\sigma_1(k)}$ to $r$. Hence $r \in [i]$ and 
$$0 \neq \langle e_{i_{\sigma_1(1)}}, \dots, e_{i_{\sigma_1(k-1)}}, \langle\mathbb  V, \dots, e_{i_{\sigma_1(k)}}, \dots,\mathbb  V \rangle_{\tau_1}, e_{i_{\sigma_1(k+1)}}, \dots, e_{i_{\sigma_1(n)}} \rangle_{\tau_2} \subset$$ $${\mathbb F} \langle e_{i_{\sigma_1(1)}}, \dots, e_{i_{\sigma_1(k-1)}}, e_r, e_{i_{\sigma_1(k+1)}}, \dots, e_{i_{\sigma_1(n)}} \rangle_{\tau_2} \cap \mathbb V \subset\mathbb V_{[i]}.$$
We have shown
\begin{equation}\label{s5}
\langle \mathbb  V_{[i]},\mathbb  V,\dots, \mathbb V \rangle_{\tau} \subset {\frak J}_{[i]}.
\end{equation}

Finally, in case there exist $i_1,\dots,i_n \in [i]$ with $v \in \mu(i_1,\dots,i_n)$ satisfying $0 \neq \langle \langle e_{i_1},\dots,e_{i_n}\rangle_{\sigma}, \mathbb V, \dots, \mathbb V, e_{j_2},...,e_{j_k} \rangle_{\tau}$ with $k \geq 2$. We have that necessarily  $0 \neq \langle \langle e_{i_1},\dots,e_{i_n} \rangle_{\sigma}, \mathbb V, \dots, \mathbb V, e_{j_2},...,e_{j_k} \rangle_{\tau} \subset {\mathbb F}e_l$ for some $l \in I$, and we have as above that $l \in [i]$. Consequently
\begin{equation}\label{sss5}
\langle \mathbb V_{[i]}, \mathbb V,\dots,\mathbb V, \mathbb W_{[i]},...,\mathbb W_{[i]} \rangle_{\tau} \subset {\frak J}_{[i]}.
\end{equation}

From Equations \eqref{s1}-\eqref{sss5} we conclude ${\frak J}_{[i]}$ is a color gLt-ideal of $\frak L$.

Finally, observe that the decomposition ${\frak J}_{[i]} = \mathbb V_{[i]} \oplus \mathbb W_{[i]}$ together with the basis $$\{e_j : j \in [i]\}$$ of $\mathbb W_{[i]}$ allow us to assert that ${\frak J}_{[i]}$ admits a
quasi-multiplicative basis inherited by the one of $\frak L$.
\end{proof}

\begin{definition}\rm
A color gLt-algebra $\frak L$ is {\it simple} if its unique non-zero color gLt-ideals are $\{0\}$ and $\frak L$.
\end{definition}

\begin{corollary}\label{co50}
If $\frak L$ is simple, then there exists a connection between any couple of elements in the index set $I$.
\end{corollary}
\begin{proof}
The simplicity of $\frak L$ applies to get that ${\frak J}_{[i_0]} =\frak L$ for any $i_0 \in I$. Hence $[i_0] = I$ and so any couple of elements in $I$ are connected.
\end{proof}

\begin{lemma}\label{lema1'}
If $[i] \neq [h]$ for some $i,h \in I$ then $\langle {\frak J}_{[i]}, {\frak J}_{[h]},\frak L, \dots, \frak L \rangle_{\sigma} = 0$ for any $\sigma \in\mathbb  S_n$.
\end{lemma}
\begin{proof}
We have to study the product $\langle\mathbb  V_{[i]} \oplus\mathbb  W_{[i]},\mathbb  V_{[h]} \oplus\mathbb  W_{[h]}, \frak L, \dots, \frak L \rangle_{\sigma}.$
By Equations \eqref{s3} and \eqref{s4} we have the following subsets of $\mathbb W_{[i]} \cap \mathbb W_{[h]}$ satisfy
\begin{equation}\label{s10}
\langle \mathbb V_{[i]},\mathbb W_{[h]},\frak L,\dots, \frak L \rangle_{\sigma} = \langle \mathbb W_{[i]}, \mathbb V_{[h]},\frak  L, \dots, \frak L \rangle_{\sigma} = \{0\}.
\end{equation}
We also have as consequence of the previous comments to Equation \eqref{s2} and to Equation \eqref{s5} that 
$$\langle\mathbb  W_{[i]},\mathbb  W_{[h]},\frak L, \dots, \frak L \rangle_{\sigma} \cap (\bigoplus\limits_{k \in I} {\mathbb F}e_k) \subset \mathbb W_{[i]} \cap\mathbb  W_{[h]} = \{0\}$$ and 
$$\langle \mathbb V_{[i]}, \mathbb V_{[h]}, \frak L, \dots, \frak L \rangle_{\sigma} \cap (\bigoplus\limits_{k \in I} {\mathbb F}e_k) \subset \mathbb W_{[i]} \cap \mathbb W_{[h]} = \{0\}$$ respectively.  From here, it just remains to consider the products $\langle e_{i'}, e_{h'}, e_{k_3}, \dots, e_{k_n} \rangle_{\sigma} \in \mathbb V$ for $i' \in [i]$, $h' \in [h]$, $k_3,\dots,k_n \in I$ and $\langle \langle e_{i'_1},\dots, e_{i'_n} \rangle_{\sigma_1}, \langle  e_{h'_1}, e_{h'_2},\dots, e_{h'_n} \rangle_{\sigma_2},\mathbb  V,\mathbb V, \ldots ,\mathbb V \rangle_{\sigma_3}\in\mathbb  V$ for $i'_1,\dots,i'_n \in [i]$, $h'_1,\dots,h'_n \in [h]$ with $\phi(i'_1, \dots, i'_n) = \phi(h'_1, \dots, h'_n) = \{v\}$. In the first situation, if $\langle e_{i'}, e_{h'}, e_{k_3}, \dots, e_{k_n} \rangle_{\sigma} \neq 0$, then  $i' \in b_{\sigma}(v, \overline{h'}, \overline{k_3},..., \overline{k_n})$ and so $i' \in \mu  (\overline{h'},v,k_3,...,k_n)$. From here
the connection $\{v,{k}_3,\dots, {k}_n\}$ gives us $h'$ is connected to $i'$, that is $[i] = [h]$ a contradiction, so $\langle e_{i'}, e_{h'}, e_{k_3}, \dots, e_{k_n} \rangle_{\sigma} = 0$.

In the second situation we deal, by Equation \eqref{gLt_condition}, with $n$-ary products of the form  $\langle e_{h'_2}, \dots, e_{h'_{k-1}}, e_{i'_{\sigma_1(1)}}, e_{h'_{k+1}}, \dots e_{h'_n} \rangle_{\sigma} $ for certain $\sigma_1 \in \mathbb S_n$ and $2 \leq k \leq n$. In case some $\langle e_{h'_2}, \dots, e_{h'_{k-1}}, e_{i'_{\sigma_1(1)}}, e_{h'_{k+1}}, \dots e_{h'_n} \rangle_{\sigma} \neq 0 $ we would have $\langle \mathbb V_{[i]}, \mathbb W_{[h]}, \dots, \mathbb W_{[h]} \rangle_{\sigma} \neq 0$ what contradicts Equation \eqref{s10}. From here any $\langle \mathbb V_{[i]}, \mathbb W_{[h]}, \dots, \mathbb W_{[h]} \rangle_{\sigma} =0$, then $\langle \langle e_{i'_1},\dots, e_{i'_n} \rangle_{\sigma_1}, \langle  e_{h'_1}, e_{h'_2},\dots, e_{h'_n} \rangle_{\sigma_2}, \mathbb V,  \ldots ,\mathbb  V \rangle_{\sigma_3}=0$ and the proof is complete. 
\end{proof}

\begin{theorem} \label{teo2}
A color gLt-algebra $\frak L = \mathbb V \oplus \mathbb W$ admitting a quasi-multiplicative basis of $\mathbb W \neq 0$ decomposes as $$\frak L = {\mathcal U} \oplus \Bigl(\sum\limits_{[i] \in I/\sim}{\frak J}_{[i]} \Bigr),$$ where ${\mathcal U}$ is a linear complement of $\sum_{[i] \in I/\sim}\mathbb  V_{[i]}$ in $\mathbb V$ and any ${\frak J}_{[i]}$ is one of the color  gLt-ideals, admitting a quasi-multiplicative basis inherited by the one of $\frak L$, given in Proposition \ref{teo1}. Furthermore $$\langle {\frak J}_{[i]}, {\frak J}_{[h]}, \frak L, \ldots, \frak L \rangle_{\sigma} = 0$$ whenever $[i] \neq [h]$.
\end{theorem}

\begin{proof}
Since we can write $$\frak L = \mathbb V \oplus \Bigl(\bigoplus\limits_{i \in I}{\mathbb F}e_i \Bigr)$$ and 
$$\mathbb V = {\mathcal U} \oplus \Bigl(\sum\limits_{[i] \in I/\sim} \mathbb V_{[i]} \Bigr), \hspace{0.3cm} \bigoplus\limits_{i \in I}{\mathbb F}e_i = \bigoplus\limits_{[i] \in I/\sim} W_{[i]}$$ we clearly have 
$$\frak L = {\mathcal U} \oplus \Bigl(\sum\limits_{[i] \in I/\sim} {\frak J}_{[i]} \Bigr)$$ being each 
${\frak J}_{[i]}$ a color gLt-ideal of $\frak L$, admitting a quasi-multiplicative basis inherited by the one of $\frak L$, satisfying $\langle {\frak J}_{[i]}, {\frak J}_{[h]}, \frak L, \dots, \frak L \rangle = 0$ when $[i] \neq [h]$ by Proposition \ref{teo1} and Lemma \ref{lema1'}.
\end{proof}

In case $\frak L$ admits a multiplicative basis (see Definition \ref{11}) we have as an immediate consequence of Theorem \ref{teo2} the next result.

\begin{corollary}\label{co20}
If $\frak L$ admits a multiplicative basis, then $$\frak L = \bigoplus\limits_{[i] \in I/\sim} {\frak J}_{[i]},$$ where any ${\frak J}_{[i]}$ is one of the color  gLt-ideals given in Proposition \ref{teo1}, admitting each one a multiplicative basis inherited by the one of $\frak L$.
\end{corollary}

\begin{definition}\rm
Let $\frak L = \mathbb V \oplus \mathbb W$ be a color gLt-algebra admitting a quasi-multiplicative basis. We call the {\it center} of $\frak L$ to the set $${\mathcal Z}(\frak L) := \{x \in\frak  L : \langle x, \frak L, \dots, \frak L \rangle_{\sigma} = 0 \hspace{0.2cm} \mbox{for any} \hspace{0.2cm} \sigma \in \mathbb S_n\}.$$ We also say that $\mathbb V$ is {\it tight} whence $\mathbb V = \{0\}$ or $\mathbb V = \sum\limits_{{\tiny
\begin{array}{c}
  i_1,\dots,i_n \in I \\
  \mu(i_1,\dots,i_n) = \{{v}\} \\
\end{array}}} {\mathbb F} \langle e_{i_1},\dots, e_{i_n} \rangle.$
\end{definition}

\begin{corollary}\label{co1}
Suppose $\frak L$ is centerless and $\mathbb V$ is tight, then $\frak L$ decomposes as the direct sum of the color gLt-ideals given in Proposition \ref{teo1}, $$\frak L = \bigoplus_{[i] \in I/\sim} {\frak J}_{[i]}.$$
\end{corollary}

\begin{proof}
By Theorem \ref{teo2}, since ${\mathcal U}=0$, we just have to show the direct character of the sum. Given $$x \in {\frak J}_{[i]}
\cap \sum\limits_{\tiny{\begin{array}{c}
[j] \in I/\sim \\
j \nsim i\\
\end{array}}} {\frak J}_{[j]},$$ by using the fact $\langle {\frak J}_{[i]}, {\frak J}_{[h]}, \frak L, \dots, \frak L \rangle_{\sigma} = 0$ for $[i] \neq [h]$ and any $\sigma \in \mathbb S_n$ we obtain $$\langle x, {\frak J}_{[i]}, \frak L, \dots, \frak L \rangle_{\sigma} = \langle x, \sum_{\tiny{\begin{array}{c}
[j] \in I / \sim  \\
j \nsim i\\
\end{array}}} {\frak J}_{[h]},\frak  L,\dots, \frak L \rangle_{\sigma} = 0.$$
It implies $\langle x, \frak L, \dots,\frak L \rangle_{\sigma} = 0,$ so $x \in {\mathcal Z}(\frak L) = 0$, as desired.
\end{proof}

\section{The minimal components}\label{minimal}

In this section we study the minimality of the components in the decompositions of color gLt-algebras given in Theorem \ref{teo2}, Corollary \ref{co20} and Corollary \ref{co1}. So we introduce the next concept.

\begin{definition}\rm
Let $\frak L = \mathbb V \oplus \mathbb W$ be a color gLt-algebra admitting a quasi-multiplicative basis ${\mathfrak B}= \{e_i\}_{i \in I}$ of $W \neq 0$. It is said that $\frak L$ is {\it minimal} if its only non-zero color gLt-ideal admitting a basis inherited by the one of $\frak L$ is itself.
\end{definition}

\noindent Let us introduce the notion of $\mu$-multiplicativity in the framework of color gLt-algebras with quasi-multiplicative bases in a similar way to the ones of closed-multiplicativity for associative quasi-multiplicative algebras, graded associative algebras, graded Lie algebras, split Leibniz algebras or split Lie triple systems (see \cite{cuasi, Yo1, Yo5, Yo4, Yo3} for these notions and examples). From now on, for any $\overline{i} \in \overline{I}$ we denote $e_{\overline{i}} = 0.$

\begin{definition}\label{defmulti}\rm
Let $\frak L = \mathbb V \oplus \mathbb W$ be a color gLt-algebra admitting a quasi-multiplicative basis ${\mathfrak B}= \{e_i\}_{i \in I}$ of $\mathbb W \neq 0$. We say that $\frak L$ admits a {\it $\mu$-quasi-multiplicative basis} if
given $i \in I$ and $k_1,\dots,k_{n} \in ({\frak I} \times \cdots \times {\frak I} ) \dot{\cup} (\overline{{\frak I}}\times \cdots \times \overline{{\frak I}})$ such that  $$i \in \mu(k_1,\dots,k_{n}) \hspace{0.2cm} {\rm then} \hspace{0.2cm} e_i \in {\mathbb F} \langle  u_{k_1}, \dots, u_{k_{n}} \rangle_{\sigma}$$ \ for some $\sigma \in S_n,$ where $u_{k_r} = e_{k_r} + e_{\overline{k}_r}$ if $k_r \notin \{v,\overline{v}\}$ or $u_{k_r} =\mathbb V$ if $k_r \in \{v, \overline{v}\},$ for $1 \leq r \leq n$.
\end{definition}

\begin{theorem}\label{last}
Suppose $\frak L$ admits a $\mu$-quasi-multiplicative basis and $\mathbb V$ is tight. It holds that $\frak L$ is minimal if and only if $I$ has all of its elements connected.
\end{theorem}

\begin{proof}
If $\frak L$ is minimal, for the color gLt-ideals defined in Proposition \ref{teo1} we have $\frak{J}_{[i]} =\frak  L$ for any $[i].$ Hence, $[i] = I$. To prove the converse, consider ${\frak J}$ a non-zero color gLt-ideal of $\frak L$ admitting a basis inherited by the one of $\frak L$. Since ${\frak J} \neq 0$, we can take some $i_0 \in I$ such that
\begin{equation}\label{ass}
0 \neq e_{i_0} \in {\frak J}.
\end{equation}
Taking into account that $I$ has all of its elements connected, we have that for any $i \in I$, we can consider a connection
\begin{equation}\label{5}
\{a_{1,2},\dots,a_{1,n},a_{2,2},\dots, a_{2,n},\dots,a_{t,2},\dots, a_{t,n}\}
\end{equation}
from $i_0$ to $i$, being $t \geq 1$. We know by Equation \eqref{bar} that $$\phi(\{\tilde{i_0}\}, a_{1,2},\dots,a_{1,n}) \cap I \neq \emptyset$$ and so for any $j_1 \in \phi(\{\tilde{i_0}\}, a_{1,2},\dots,a_{1,n}) \cap I$ we have, taking into account $\tilde{i_0} \in \{i_0, \overline{i_0}\}$ that either $j_1 \in \mu (i_0, a_{1,2},..., a_{1,n}) \setminus \{v\}$ or $j_1 \in \mu (\overline{i_0}, a_{1,2},..., a_{1,n}) \setminus \{v\}$, being necessarily any $a_{1,k} \in {\frak I}$ in the second possibility. By Equation (\ref{ass}) we get in the first possibility that $0 \neq e_{j_1} \in {\mathbb F} \langle  e_{i_0},u_{a_{1,2}}, \dots, u_{a_{1,n}} \rangle_{\sigma} \subset {\frak I}$ for some $\sigma \in \mathbb S_n$, and with $u_{a_{1,k}}=e_{a_{1,k}} + \overline{e_{a_{1,k}}}$ if $a_{1,k} \in I \cup \overline{I}$ or $u_{a_{1,k}}=\mathbb V$ if $a_{1,k} \in \{v, \overline{v}\}.$ In the second possibility, we get by
Equation (\ref{ass}) and the $\mu$-quasi-multiplicativity of ${\mathfrak B}$ that $e_{j_1} \in {\mathbb F} \langle  e_{i_0},u_{a_{1,2}}, \dots, u_{a_{1,n}} \rangle_{\sigma} \subset {\frak I}$ for some $\sigma \in\mathbb S_n$, where
$u_{a_{1,k}}=e_{a_{1,k}}$ if $a_{1,k} \in I$ or $u_{a_{1,k}}=\mathbb V$ if $a_{1,k} =v.$

 Hence we can assert
\begin{equation}\label{7}
\bigoplus\limits_{j \in \phi(\{i_0\}, a_{1,2}, \dots, a_{1,n}) \cap I} {\mathbb F} e_j\subset {\frak J}.
\end{equation}
Since $$\phi(\phi(\{i_0\},a_{1,2},\dots,a_{1,n}),a_{2,2},\dots,a_{2,n}) \cap I \neq \emptyset,$$ we can argue as above, taking into account Equation \eqref{7}, to get
$$\bigoplus\limits_{j \in \phi(\phi(\{i_0\},a_{1,2},\dots,a_{1,n}),a_{2,2},\dots,a_{2,n})\cap I} {\mathbb F} e_j \subset {\frak J}.$$
By reiterating this process with the connection \eqref{5} we obtain $$\bigoplus\limits_{j \in \phi(\phi (\ldots \phi(\{i_0\}, a_{1,2}, \dots, a_{1,n}), \ldots), a_{t,2},\dots, a_{t,n}) \cap I} {\mathbb F} e_j \subset {\frak J}.$$ Taking now into account $i \in \phi(\phi (\ldots \phi(\{i_0\},a_{1,2},\dots,a_{1,n}), \ldots), a_{t,2},\dots, a_{t,n}) \cap I$ we conclude $e_i \in {\frak J}$ and so
\begin{equation}\label{b1}
\mathbb W = \bigoplus\limits_{i \in I} {\mathbb F}e_i \subset {\frak J}.
\end{equation}
Taking now into account that $\mathbb V$ is tight, Equation \eqref{b1} allows us to assert
\begin{equation}\label{b2}
\mathbb V \subset {\frak J}.
\end{equation}
Finally, since $\frak L = \mathbb V \oplus \mathbb W$, Equations \eqref{b1} and \eqref{b2} give us ${\frak J}=\frak L$.
\end{proof}

\begin{theorem}
Suppose $\frak L$ admits a $\mu$-quasi-multiplicative basis. If $\frak L$ is centerless and with $\mathbb V$ tight then 
$$\frak L = \bigoplus\limits_{k} {\frak J}_k$$ is the direct sum of the family of its minimal color gLt-ideals, each one admitting a $\mu$-quasi-multiplicative basis inherited by the one of $\frak L$.
\end{theorem}

\begin{proof}
By Corollary \ref{co1} we have that $\frak L = \bigoplus_{[i] \in I/\sim} {\frak J}_{[i]}$ is the direct sum of the color gLt-ideals ${\frak J}_{[i]}.$

\noindent We wish to apply Theorem \ref{last} to any ${\frak J}_{[i]}$, so we have to verify that $${\frak J}_{[i]} =\mathbb V_{[i]} \oplus\mathbb W_{[i]}$$ admits a $\mu$-quasi-multiplicative basis, $\mathbb V_{[i]}$ is tight
and the basis $\{e_i : i \in [i]\}$ of $\mathbb W_{[i]}$ satisfies that all of the elements in the index set $[i]$ are $[i]$-connected (connected through connections contained in $([i] \dot{\cup} v)
\dot{\cup}(\overline{[i]} \dot{\cup} \overline{v})$).

We clearly have that ${\frak J}_{[i]}$ admits a $\mu$-quasi-multiplicative basis as consequence of having a basis inherited from the one of $\frak L$ and that the linear space $\mathbb V_{[i]} $ is tight by construction.

Finally, since it is easy to verify that $[i]$ has all of its elements $[i]$-connected we can apply Theorem \ref{last} to any ${\frak J}_{[i]}$ so as to conclude ${\frak J}_{[i]}$ is minimal. It is clear that the decomposition $\frak L = \bigoplus_{[i] \in
I/\sim} {\frak J}_{[i]}$ satisfies the assertions of the theorem.
\end{proof}

\medskip


\end{document}